\documentclass
{amsart}
\usepackage{graphicx}
\newtheorem{thm}{Theorem}[section]
\theoremstyle{definition}
\newtheorem{cor}[thm]{Corollary}
\newtheorem{prop}[thm]{Proposition}
\newtheorem{defn}[thm]{Definition}
\newtheorem{lem}[thm]{Lemma}

\newtheorem{rem}[thm]{Remark}

\numberwithin{equation}{section}
\begin{document}
\title[Quasi $z^\circ$-submodules of a reduced multiplication module]
{Quasi $z^\circ$-submodules of a reduced multiplication module}
\author{F.  Farshadifar}
\address{Department of Mathematics, Farhangian University, Tehran, Iran.}
\email{f.farshadifar@cfu.ac.ir}


% ----------------------------------------------------------------
\begin{abstract}
Let $R$ be a commutative ring with identity and $M$ be an $R$-module.
The purpose of this paper is to defined the notion of quasi $z^\circ$-submodules of $M$ as an extension of $z^\circ$-ideals of $R$ and obtained some related results when $M$ is a reduced multiplication $R$-module.
\end{abstract}
\maketitle

\section{Introduction}
\noindent
Throughout this paper, $R$ will denote a commutative ring with
identity.
An $R$-module $M$ is said to be a \emph{multiplication module} if for every submodule $N$ of $M$ there exists an ideal $I$ of $R$ such that $N=IM$ \cite{Ba81}.

A proper ideal $I$ of $R$ is called a \textit{$z$-ideal} whenever any two
elements of $R$ are contained in the same set of maximal ideals and $I$ contains
one of them, then it also contains the other one \cite{MR321915}.
For each $a \in R$, let $\mathfrak{P}_a$, be the intersection of all minimal prime ideals of $R$ containing $a$.
A proper ideal $I$ of $R$ is called a $z^\circ$-ideal if for each $a \in I$
we have $\mathfrak{P}_a \subseteq I$ \cite{MR1736781}.

For a submodule $N$ of an $R$-module $M$, let $\mathcal{M}(N)$ be the set of maximal submodules of $M$ containing
$N$ and $Max(M)$ be the set of all maximal submodules of $M$. The intersection of all maximal submodules of $M$ containing $N$ is said to be the \textit{Jacobson radical} of $N$ and denote by $Rad_N(M)$ \cite{MR3677368}. In case $N$ does not contained in
any maximal submodule, the Jacobson radical of $N$ is defined to be $M$. We denote the Jacobson radical of zero submodule of $M$ by $Rad_M(M)$. A proper submodule $N$ of $M$ is said to be a \textit{$z$-submodule} if for every $x,y \in M$, $\mathcal{M}(x)=\mathcal{M}(y)\not= \emptyset$ and $x \in N$ imply $y \in N$ \cite{F401}.

Let $M$ be an $R$-module. A proper submodule $P$ of $M$ is said to be \textit{prime} if for any
$r \in R$ and $m \in M$ with $rm \in P$, we have $m \in P$ or $r \in (P :_R M)$. In this case, $(P :_R M)$ is a prime ideal of $R$ \cite{MR183747, MR498715}.
 $M$ is said to be \textit{reduced} if the intersection of all prime submodules of $M$
is equal to zero \cite{MR2839935}.
The intersection of all prime submodules of $M$ containing a submodule $N$ of $M$ is said to be the \textit{prime
radical} of $N$ and denote by $rad_NM$. In case $N$ does not contained in
any prime submodule, the prime radical of $N$ is defined to be $M$ \cite{MR824879}.
A prime submodule $P$ of $M$
is a minimal prime submodule over $N$ if $P$ is a minimal element of the set of all
prime submodules of $M$ that contain $N$. \textit{A minimal prime submodule} of $M$ means a minimal prime submodule over the $0$ submodule of $M$.
The set of all minimal prime submodules of $M$ will be denoted by $Min(M)$. The intersection of all minimal prime submodules of $M$ containing a submodule $K$ of $M$ is denote by $\mathfrak{P}_K$. If $N$ is a submodule of $M$, define
 $V (N) = \{P \in Min(M) : N \subseteq P\}$.

In \cite{F402}, the notion of $z^\circ$-submodules of an $R$-module $M$ as an extension of $z^\circ$-ideals was introduced and some of their properties when $M$ is a reduced multiplication $R$-module deals. A proper submodule $N$ of an $R$-module $M$ is said to be a \textit{$z^\circ$-submodule} of $M$ if  $\mathfrak{P}_x \subseteq N$  for all  $x \in N$ \cite{F402}.
In this paper, we introduced the notion of quasi $z^\circ$-submodules of an $R$-module $M$ as another generalization of $z^\circ$-ideals. Also, we investagate some related results when $M$ is a reduced multiplication $R$-module.
%%%%%%%%%%%%%%%%%%%%%%%%%
%%%%%%%%%%%%%%%%%%%%%%%%%%%%%%%%%
%%%%%%%%%%%%%%%%%%%%%%%%%%%%%%%%%%%%%%%%
%%%%%%%%%%%%%%%%%%%%%%%%%%%%%%%%%%%%%%%%%%%%%%%%
\section{Reduced multiplication modules}
\begin{lem}\label{ll1.4}
Let $M$ be a reduced multiplication $R$-module and $P$ be a minimal prime submodule of $M$.
If $a \in (P:_RM)$, then
$$
Ann_{R/Ann_R(M)}(a+Ann_R(M)) \not \subseteq  (P:_RM)/Ann_R(M).
$$
\end{lem}
\begin{proof}
By \cite[Proposition 1.5]{MR2378535}, we have $(P:_RM)$ is a minimal prime ideal of $R$ over $Ann_R(M)$.
As $M$ is a reduced multiplication $R$-module, $R/Ann_R(M)$ is a reduced ring by \cite{MR2839935}.
Let  $a+Ann_R(M) \in (P:_RM)/Ann_R(M)$.
 Then by \cite[Proposition 1.2 (1)]{MR698302}, we have $Ann_{R/Ann_R(M)}(a+Ann_R(M)) \not \subseteq  (P:_RM)/Ann_R(M)$.
\end{proof}

\begin{rem}\label{rr1.11}\cite[Remark ]{F402}
Let $M$ be a multiplication $R$-module and $\Omega$ be a subset of $Min(M)$. Set $\mathfrak{P}_{\Omega}=\cap \{P: P \in \Omega\}$.
A subset $\Omega $ of $Min(M)$ is said to be closed if $\Omega=V(\mathfrak{P}_{\Omega})$. With this notion of
closed set, one can see that the space of minimal prime submodules of $M$ becomes a topological space.
\end{rem}

\begin{thm}\label{tt1.11}
Let $M$ be a faithful reduced multiplication $R$-module. Then for each $a \in R$,  we have $V(Ann_R(a)M)=Min(M) \setminus V(aM)$. In particular, $V(Ann_R(a)M)$ and $V(aM)$ are disjoint open-and-closed sets.
\end{thm}
\begin{proof}
If $P\in V(aM)$, then by Lemma \ref{ll1.4}, $Ann_R(a) \not\subseteq (P:_RM)$ and so $Ann_R(a)M  \not\subseteq P$. Thus $V(Ann_R(a)M)\cap V(aM)=\emptyset$. On the other
hand, if $P\in Min(M) \setminus V(aM)$, then for any $b \in Ann_R(a)$, we have $abM=0 \subseteq P$. Since
$a \not \in (P:_RM)$ and $P$ is prime, $bM \subseteq P$. Therefore, $P \in V(Ann_R(a)M)$.  Thus $V(Ann_R(a)M)=Min(M) \setminus V(aM)$. Both sets $V(Ann_R(a)M)$ and $V(aM)$ are closed, and
since they are complementary, they are also open.
\end{proof}

\begin{cor}\label{cc1.11}
Let $M$ be a faithful reduced multiplication $R$-module.  Then
 $Min(M)$ is a Hausdorff space with a base of open-and-closed sets.
\end{cor}
\begin{proof}
Let $P \not= \acute{P} \in Min(M)$. Assume that $a \in (P:_RM)\setminus (\acute{P}:_RM)$. Then $V(aM)$ and $V(Ann_R(a)M)$ are
disjoint open sets containing $P $ and $\acute{P}$, respectively. Hence $Min(M)$ is a Hausdorff
space. In fact, the family $\{V(aM)\}$ is a base for the closed
sets. Thus $V(Ann_R(a)M)$ is a base for the open sets.
\end{proof}

\begin{thm}\label{t1.5}
Let $M$ be a reduced multiplication $R$-module. Then we have the following.
\begin{itemize}
\item [(a)] If $M$ is a faithful $R$-module, then $V(aM)=V((0:_MAnn_R(aM)))$ for each  $a \in R$.
\item [(b)] $(0:_MAnn_R(IJM))=(0:_MAnn_R(IM)) \cap (0:_MAnn_R(JM))$ for each ideals $I, J$ of $R$.
\end{itemize}
\end{thm}
\begin{proof}
(a)  Let $M$ be a faithful $R$-module and $a \in R$.
As $aM\subseteq (0:_MAnn_R(aM))$, we have $V((0:_MAnn_R(aM)))\subseteq V(aM)$. Now let
$P$ be a minimal prime submodule of $M$ containing $aM$. Then there exists
 $b \in Ann_R(aM) \setminus (P:_RM)$ by Lemma \ref{ll1.4}. Then, for any $y \in (0:_MAnn_R(aM))$, we have $by=0 \in P$. Hence $y \in P$. So, $(0:_MAnn_R(aM))\subseteq P$, as needed.

(b) Let $I, J$ be ideals of $R$. As $Ann_R(IM) \subseteq Ann_R(IJM)$ and $Ann_R(JM) \subseteq Ann_R(IJM)$, we have
$$
 (0:_MAnn_R(IJM))\subseteq (0:_MAnn_R(JM)) \cap (0:_MAnn_R(IM)).
$$
Now suppose that $z \in (0:_MAnn_R(JM)) \cap (0:_MAnn_R(IM))$. Let $t \in Ann_R(IJM)$. Then $Jt \subseteq Ann_R(IM)$ and so $tJz=0$. Since $M$ is a multiplication module, $tzR=AM$ for some ideal $A$ of $R$. Hence, $A \subseteq Ann_R(JM)$ and so $Az=0$. Thus $t(Rz)^2=0$. It follows that $(Rtz)^2=0$. This implies that $tz=0$ since $M$ is a reduced multiplication module. Hence, $zAnn_R(IJM)=0$. It follows that $z \in (0:_MAnn_R(IJM))$.
\end{proof}

\begin{cor}\label{c1.13}
Let $M$ be a faithful reduced multiplication $R$-module. Then for each $a \in R$, $(0:_MAnn_R(aM))=\mathfrak{P}_{aM}$.
\end{cor}
\begin{proof}
Since by \cite[Theorem 2.11]{F402}, $(0:_MI)=\mathfrak{P}_{(0:_MI)}$ for each ideal $I$ of $R$. The result follows from Theorem \ref{t1.5} (a).
\end{proof}

\begin{thm}\label{t4.6}
Let $M$ be a reduced multiplication $R$-module. Then the following are equivalent:
\begin{itemize}
\item [(a)] For $a, b\in R$, $Ann_R(aM)=Ann_R(bM)$ and $aM \subseteq N$ imply that $bM \subseteq N$;
\item [(b)] For $a, b\in R$, $Ann_R(aM)\subseteq Ann_R(bM)$ and $aM \subseteq N$ imply that $bM \subseteq N$.
\end{itemize}
\end{thm}
\begin{proof}
$(a)\Rightarrow (b)$
Let for $a, b\in R$, $Ann_R(aM)\subseteq Ann_R(bM)$ and $aM \subseteq N$. Then $(0:_MAnn_R(bM)) \subseteq (0:_MAnn_R(aM))$. Hence
$(0:_MAnn_R(abM))=(0:_MAnn_R(bM))$ by Theorem \ref{t1.5} (b). It follows that $Ann_R(abM)=Ann_R(bM)$. Now as $abM \subseteq N$ we have $bM \subseteq N$ by part (a).

$(b)\Rightarrow (a)$
This is clear.
\end{proof}

\begin{thm}\label{t119.3}
Let a faithful multiplication $R$-module. Then $\mathfrak{P}_IM \subseteq \mathfrak{P}_{IM}$ for each ideal $I$ of $R$. The reverse inclusion holds when $M$ is a finitely generated $R$-module.
\end{thm}
\begin{proof}
Let $I$ be an ideal of $R$ and $X$ be a minimal prime submodule of $M$ such that $IM \subseteq X$. Then $I \subseteq (X:_RM)$. Since by \cite[Proposition 1.5]{MR2378535},  $(X:_RM)$ is a minimal prime ideal of $R$, $\mathfrak{P}_I \subseteq (X:_RM)$. Thus  $\mathfrak{P}_IM \subseteq (X:_RM)M \subseteq X$. Hence,
$\mathfrak{P}_IM \subseteq \mathfrak{P}_{IM}$ for each ideal $I$ of $R$. For the converse, let $\mathfrak{P}_I=\cap \mathfrak{p}_i$, where $\mathfrak{p}_i \in Min(R), I \subseteq \mathfrak{p}_i$, and $Min(R)$ is the set of all minimal prime ideals of $R$. As $M$ is a faithful multiplication $R$-module, by using \cite[Theorem 1.6]{MR932633},
$$
IM\subseteq \mathfrak{P}_IM=(\bigcap \mathfrak{p}_i)M =
\bigcap \mathfrak{p}_iM=\bigcap_{\mathfrak{p}_iM \not=M}\mathfrak{p}_iM.
$$
This implies that $\mathfrak{P}_{IM}\subseteq \mathfrak{P}_IM$ since $M$ is finitely generated and so by using \cite[Page 762]{MR932633}, $\mathfrak{p}_iM\not=M$ is a minimal prime submodule of $M$.
\end{proof}
%%%%%%%%%%%%%%%%%%%%%%%%%
%%%%%%%%%%%%%%%%%%%%%%%%%%%%%%%%%
%%%%%%%%%%%%%%%%%%%%%%%%%%%%%%%%%%%%%%%%
%%%%%%%%%%%%%%%%%%%%%%%%%%%%%%%%%%%%%%%%%%%%%%%%
\section{Quasi $z^\circ$-submodules}
\begin{defn}\label{d1.1}
We say that a proper submodule $N$ of an $R$-module $M$ is a \textit{quasi $z^\circ$-submodule} of $M$ if  $\mathfrak{P}_{aM} \subseteq N$  for all  $a \in (N:_RM)$.
\end{defn}

\begin{rem}\label{r1.1}
Let $M$ be an $R$-module.
If $N$ is a quasi $z^\circ$-submodule of $M$, then for
each $a \in (P:_RM)$ we have  $\mathfrak{P}_{aM}\not=M$, i.e. $aM$ contained at least in a minimal prime submodule of $M$.   Clearly, every minimal prime submodule of $M$ is a quasi $z^\circ$-submodule of $M$. Also, the family of
quasi $z^\circ$-submodules of $M$ is closed under intersection. Therefore, if  $\mathfrak{P}_{0}\not=M$, then $\mathfrak{P}_{0}$ is a quasi $z^\circ$-submodule of $M$ and it is contained in every quasi $z^\circ$-submodule of $M$.
\end{rem}

We give the following easy results without proofs.
\begin{prop}\label{pr551.1}
Let $N$ be a submodule of a cyclic $R$-module $M$. Then $N$ is a $z^\circ$-submodule of $M$ if and only if $N$ is a quasi $z^\circ$-submodule of $M$.
\end{prop}

\begin{lem}\label{l1.1}
Let $M$ be an $R$-module.
A submodule $N$ of $M$ is a quasi $z^\circ$-submodule if and only if $N=\sum_{a\in (N:_RM)} \mathfrak{P}_{aM}$.
\end{lem}

\begin{prop}\label{p1000.1}
Let $N$ be a proper submodule of an $R$-module $M$. Then $N$ as an $R$-submodule is a quasi $z^\circ$-submodule if and only if as an $R/Ann_R(M)$-submodule is a quasi $z^\circ$-submodule.
\end{prop}

\begin{lem}\label{t19.3}
Let $M$ be a faithful multiplication $R$-module $M$. If $N$ is a quasi $z^\circ$-submodule $M$, then $(N:_RM)$ is a $z^\circ$-ideal of $R$. The converse holds when $M$ is a finitely generated $R$-module.
\end{lem}
\begin{proof}
This follows from Theorem \ref{t119.3}.
\end{proof}

\begin{thm}\label{c19.3}
Let $M$ be a faithful multiplication $R$-module. Then we have the following.
\begin{itemize}
\item [(a)] Let $M$ be a finitely generated $R$-module and $I$ be a $z^\circ$-ideal of $R$. Then $IM$ is a quasi $z^\circ$-submodule of $M$.
\item [(b)] Let $R$ be a reduced ring and $N$ be a quasi $z^\circ$-submodule of an $R$-module $M$. Then $(N:_R(K:_RM)M)$ is a  $z^\circ$-ideal of $R$ for each submodule $K$ of $M$. In particular, if $\mathfrak{P}_0=0$, then $Ann_R((K:_RM)M)$ is a $z^\circ$-ideal of $R$ for each submodule $K$ of $M$.
\item [(c)] Let $R$ be a reduced ring and $N$ be a quasi $z^\circ$-submodule of a multiplication $R$-module $M$. Then $(N:_RK)$ is a $z^\circ$-ideal of $R$ for each submodule $K$ of $M$. In particular, if $\mathfrak{P}_0=0$, then $Ann_R(K)$ is a $z^\circ$-ideal of $R$ for each submodule $K$ of $M$.
\end{itemize}
 \end{thm}
\begin{proof}
(a) By \cite[Theorem 10]{MR933916}, $I=(IM:_RM)$. Now the result follows from Lemma \ref{t19.3}.

(b) As $N$ is a quasi $z^\circ$-submodule, $(N:_RM)$ is a $z^\circ$-ideal of $R$ by Lemma \ref{t19.3}. Let $K$ be a submodule of $M$. Then by \cite[Examples of $z^\circ$-ideals]{MR1736781}, $((N:_RM):_R(K:_RM))$ is a $z^\circ$-ideal of $R$. Now $(N:_R(K:_RM)M)=((N:_RM):_R(K:_RM))$ implies that $(N:_R(K:_RM)M)$ is a $z^\circ$-ideal of $R$. Now the last assertion follows from the fact that $\mathfrak{P}_0$ is a quasi $z^\circ$-submodule of $M$ by Remark \ref{r1.1}.

(c) As $M$ is a multiplication $R$-module, $K=(K:_RM)M$. Now the result follows from part (b)
\end{proof}

\begin{cor}\label{cc19.3}
Let $M$ be a faithful cyclic $R$-module. Then $N$ is a $z^\circ$-submodule of an $R$-module $M$ if and only if $(N:_RM)$ is a $z^\circ$-ideal of $R$.
\end{cor}
\begin{proof}
This follows from Lemma \ref{t19.3} and the fact that every cyclic $R$-module is a multiplication $R$-module.
\end{proof}

Let $M$ be an $R$-module. The set of torsion elements
of $M$ with respect to $R$ is $T_0(M) = \{m \in M |rm = 0 \ for \ some\ 0 \not= r \in R\}$.
\begin{thm}\label{p1.11}
Let $M$ be a faithful reduced multiplication $R$-module, $a , b \in M$, and $b \in Ann_R(aM)$.
If $Ann(aM)M = \mathfrak{P}_{bM}$, then $(a +b)M \not\subseteq T_0(M)$. The converse holds when $M$ is faithful and $Ann_R(a)M=Ann_R(aM)M$ is a quasi $z^\circ$-submodule of $M$.
\end{thm}
\begin{proof}
First note that $T_0(M)=\cup_{P \in Min(M)}P$ by \cite[Corollary 2.9]{F402}.
Now let $Ann(aM)M = \mathfrak{P}_{bM}$ and $(a +b)M$ belong to a minimal prime submodule $P$ of
$M$ and seek a contradiction. We consider two cases. First, let $a \in (P:_RM)$. Then $b \in (P:_RM)$ implies that $\mathfrak{P}_{bM}\subseteq P$, i.e., $\mathfrak{P}_{bM} =Ann_R(aM)M \subseteq P$, which is a contradiction by Theorem \ref{tt1.11}. Now let $a \not \in (P:_RM)$, then we must have $b \not \in (P:_RM)$, i.e.,
$Ann(aM)M = \mathfrak{P}_{bM}\not \subseteq P$, which is impossible by Theorem \ref{tt1.11}.. Conversely, if $(a +b)M \not \subseteq T_0(M)$, we are to show that $Ann(aM)M \subseteq \mathfrak{P}_{bM}$. Let $P$ be a minima1
prime submodule with $b \in (P:_RM)$, then $a + b \not \in (P:_RM)$ implies that $a \not \in (P:_RM)$, i.e., $Ann(aM)M \subseteq P$.
Hence $Ann(aM)M \subseteq \mathfrak{P}_{bM}$. The reverse inclusion follows from the fact that $Ann_R(aM)M$ is a quasi $z^\circ$-submodule of $M$.
\end{proof}

Let $M$ be an $R$-module. The set of zero divisors of $R$
on $M$ is $Zd_R(M)= \{r \in R| rm=0 \ for \ some\ nonzero\ m \in M\}$.
\begin{prop}\label{p1.1}
Let $M$ be a faithful multiplication $R$-module. If $N$ is a quasi $z^\circ$-submodule of $ M$. Then $(N:_RM) \subseteq Zd_R(M)$.
\end{prop}
\begin{proof}
By \cite[Lemma 2.1]{MR3755273}, $Zd_R(R)=Zd_R(M)$. Now the result follows from the fact that $(N:_RM)$ is a $z^\circ$-ideal of $R$ by Lemma \ref{t19.3}.
\end{proof}

\begin{thm}\label{t1.6}
Let $M$ be a faithful reduced multiplication $R$-module. Then the following are equivalent:
\begin{itemize}
\item [(a)] $N$ is a quasi $z^\circ$-submodule of $M$;
\item [(b)] For each $a, b  \in (N:_RM)$, $\mathfrak{P}_{aM}=\mathfrak{P}_{bM}$ and $aM \subseteq  N$ imply that $bM \subseteq N$;
\item [(c)] For each $a, b  \in (N:_RM)$, $V(aM)=V(bM)$ and $aM \subseteq N$ imply that $bM \subseteq N$;
\item [(d)] For each ideal $a \in R$, we have $aM \subseteq N$ implies that $(0:_MAnn_R(aM)) \subseteq N$;
\item [(e)] For $a, b \in R$, $Ann_R(aM)=Ann_R(bM)$ and $aM \subseteq N$ imply that $bM \subseteq N$;
\item [(f)] For $a, b\in R$, $Ann_R(aM)\subseteq Ann_R(bM)$ and $aM \subseteq N$ imply that $bM \subseteq N$.
\end{itemize}
\end{thm}
\begin{proof}
$(a)\Rightarrow (b)$
Let for $a, b  \in (N:_RM)$, $\mathfrak{P}_{aM}=\mathfrak{P}_{bM}$ and $aM \subseteq  N$. By part (a), $\mathfrak{P}_{aM}\subseteq N$. Thus $bM \subseteq \mathfrak{P}_{bM} \subseteq N$.

$(b)\Rightarrow (c)$
Let for $a, b  \in (N:_RM)$, $V(aM)=V(bM)$ and $aM \subseteq N$. Then  $\mathfrak{P}_{aM}=\mathfrak{P}_{bM}$. Thus by part (b), $bM \subseteq N$.

$(c)\Rightarrow (d)$
Let $aM \subseteq N$. Then  $V(aM)=V((0:_MAnn_R(aM))$  by Theorem \ref{t1.5}. Thus by part (c), $(0:_MAnn_R(aM)) \subseteq N$.

$(d)\Rightarrow (e)$
Let for $a, b  \in (N:_RM)$, $Ann_R(aM)=Ann_R(bM)$ and $aM \subseteq N$. Then $(0:_MAnn_R(aM))=(0:_MAnn_R(bM))$. By part (d), $(0:_MAnn_R(aM)) \subseteq N$. Thus $bM \subseteq (0:_MAnn_R(bM))\subseteq N$.

$(e)\Rightarrow (f)$
By Theorem \ref{t4.6}.

$(f)\Rightarrow (a)$
Let $aM \subseteq N$. By Corollary \ref{c1.13}, $(0:_MAnn_R(aM))=\mathfrak{P}_{aM}$. Let $x \in (0:_MAnn_R(aM))$. Then $Ann_R(aM) \subseteq Ann_R(x)$. As $M$ is a multiplication $R$-module, $Rx=JM$ for some ideal $J$ or $R$. Let $b \in J$. Then $Ann_R(JM) \subseteq Ann_R(bM)$. Therefore, $Ann_R(aM) \subseteq Ann_R(bM)$. Thus by part (f), $bM \subseteq N$ and so $Rx=JM \subseteq N$. This implies that $\mathfrak{P}_{aM}=(0:_MAnn_R(aM))\subseteq N$.
\end{proof}

An $R$-module $M$ is said to be a \emph{comultiplication module} if for every submodule $N$ of $M$ there exists an ideal $I$ of $R$ such that $N=(0:_MI)$ equivalently, for each submodule $N$ of $M$, we have $N=(0:_MAnn_R(N))$  \cite{MR3934877}.
\begin{cor}\label{c1.7}
Every proper submodule of a faithful reduced multiplication and comultiplication $R$-module is a quasi $z^\circ$-submodule.
\end{cor}
\begin{proof}
As $M$ is a comultiplication $R$-module, for each $a \in M$ we have $aM=(0:_MAnn_R(aM))$. Now the result follows from Theorem \ref{t1.5} $(d)\Rightarrow (a)$.
\end{proof}

\begin{thm}\label{t7.9}
Let $M$ be a faithful reduced multiplication $R$-module and $N$ be a quasi $z^\circ$-submodule of $M$. Then every
prime submodule, minimal over $N$ is a prime quasi $z^\circ$-submodule of $M$.
\end{thm}
\begin{proof}
Let $P$ be a prime submodule, minimal over $N$. Assume that $Ann(aM) \subseteq Ann(bM)$, where
$a\in (P:_RM)$ and $b \in R$. Since $P/N$ is a minimal prime submodule of $M/N$,
By Lemma \ref{ll1.4} (b), there exists $c \in Ann_R(a(M/N)) \setminus (P/N:_RM/N)$. Thus $ca \in (N:_RM)$ and $c \not \in (P:_RM)$. Now we have  $Ann(caM) \subseteq Ann(cbM)$. As $N$ is a quasi $z^\circ$-submodule of $M$, we get that
$cb \in (N:_RM)\subseteq (P:_RM)$. As $c \not \in (P:_RM)$ and $P$ is a prime submodule, $b \in (P:_RM)$, as needed.
\end{proof}

\begin{cor}\label{c7.10}
If $f: M\rightarrow M/N$ is the natural epimorphism, where $M$ is
a faithful reduced multiplication $R$-module and $N$ is a quasi $z^\circ$-submodule of $M$, then every quasi $z^\circ$-submodule of $M/N$ contracts to
a quasi $z^\circ$-submodule of $M$.
\end{cor}

\begin{cor}\label{c7.12}
If $M$ is a faithful reduced multiplication $R$-module, then every maximal quasi $z^\circ$-submodule is a
prime  quasi $z^\circ$-submodule.
\end{cor}

\begin{cor}\label{c7.13}
Let $M$ be a faithful reduced multiplication $R$-module and $P$ be a prime submodule of $M$. Then either $P$ is a quasi $z^\circ$-submodule or contains a maximal  quasi $z^\circ$-submodule which is a prime
 quasi $z^\circ$-submodule.
\end{cor}

\begin{thm}\label{t4.8}
Let $M$ be a faithful multiplication $R$-module. Then the following are equivalent:
\begin{itemize}
\item [(a)] $M$ is a reduced module, i.e., $R$ is a reduced ring;
\item [(b)] The submodule $0$ is a quasi $z^\circ$-submodule of $M$.
\end{itemize}
\end{thm}
\begin{proof}
$(a)\Rightarrow (b)$
Let $x=0$. Then $(0:_MAnn_R(x))=0$. Thus the result follows from Theorem \ref{t1.6} $(d)\Rightarrow (a)$.

$(b)\Rightarrow (a)$
Let $a \in R$ such that $(Ra)^2=0$. It is clear that $\mathfrak{P}_{aM}=\mathfrak{P}_{(a^2M)}$. Thus
$\mathfrak{P}_{aM}=\mathfrak{P}_{a^2M}=\mathfrak{P}_{0}$. Since the submodule $0$ is a $z^\circ$-submodule, $aM=0$. Now as $M$ is faithful, $a=0$.
\end{proof}

\end{document}